\documentclass[a4paper,12pt]{amsart}

\usepackage{amsfonts,amsmath,hyperref,color}

\theoremstyle{proclaim}
\newtheorem{theorem}{Theorem}[section]
\newtheorem{lemma}[theorem]{Lemma}

\theoremstyle{statement}

\begin{document}

\title{An estimate for narrow operators on $L^p([0, 1])$}

\author{Eugene Shargorodsky {\protect \and} Teo Sharia}
\address{E. Shargorodsky \\
Department of Mathematics\\
King's College London\\
Strand, London WC2R 2LS\\
United Kingdom\quad
and\quad
Technische Universit\"at Dresden\\
Fakult\"at Mathematik\\
01062 Dresden\\
Germany}
\email{eugene.shargorodsky@kcl.ac.uk}

\address{T. Sharia\\
Department of Mathematics\\
Royal Holloway\\
University of London\\
Egham, Surrey TW20 0EX\\
United Kingdom}
\email{t.sharia@rhul.ac.uk}

\date{}

\begin{abstract}
We prove a theorem, which generalises C. Franchetti's estimate for the norm of a projection onto a rich subspace of
$L^p([0, 1])$ and the authors' related estimate for compact operators on $L^p([0, 1])$, $1 \le p < \infty$.
\end{abstract}


\maketitle

\section{Introduction}

For Banach spaces $X$ and $Y$, let $\mathcal{B}(X, Y)$ and $\mathcal{K}(X, Y)$ denote the sets of bounded linear and compact linear operators 
from $X$ to $Y$, respectively; $\mathcal{B}(X) := \mathcal{B}(X, X)$, $\mathcal{K}(X) := \mathcal{K}(X, X)$; $I \in \mathcal{B}(X)$ denotes the identity operator.
An operator $P \in \mathcal{B}(X)$ is called a \textit{projection} if $P^2 = P$. A closed linear subspace $X_0 \subset X$ is called \textit{1-complemented} (in $X$) if there exists a
projection $P \in \mathcal{B}(X)$ such that $P(X) = X_0$ and $\|P\| = 1$.

Let $(\Omega, \Sigma, \mu)$ be a nonatomic measure space with $0 < \mu(\Omega) < \infty$. We will use the following
notation:
\begin{itemize}
 \item $\Sigma^+ :=\{A \in \Sigma : \ \mu(A) > 0\}$,
  \item $\mathbb{I}_A$ is the indicator function of $A \in \Sigma$, i.e. $\mathbb{I}_A(\omega) = 1$ if $\omega \in A$ and 
$\mathbb{I}_A(\omega) = 0$ if $\omega \not\in A$,
  \item $\mathbf{1} := \mathbb{I}_\Omega$,
  \item $\mathbf{ E}f := \left(\frac{1}{\mu(\Omega)} \int_\Omega f\, d\mu\right) \mathbf{1}$.
\end{itemize}

We will use the terminology from \cite{PR13}. 
A $\Sigma$-measurable function $g$ is called a \textit{sign} if it takes values in the set $\{-1, 0, 1\}$, and a \textit{sign on} $A \in \Sigma$ if it is
a sign with the support equal to $A$, i.e. if $g^2 = \mathbb{I}_A$. A sign is of \textit{mean zero} if $\int_\Omega g\, d\mu = 0$. 

An operator $T \in \mathcal{B}(L^p(\mu), Y)$, $1 \le p < \infty$ is called \textit{narrow} if for every $A \in \Sigma^+$ and every $\varepsilon > 0$,
there exists a mean zero sign $g$ on $A$ such that $\|Tg\| < \varepsilon$.

Every $T \in \mathcal{K}(L^p(\mu), Y)$ is narrow (see \cite[Proposition 2.1]{PR13}), but there are noncompact narrow operators. 
Indeed, let $\mathcal{G}$ be a sub-$\sigma$-algebra of $\Sigma$
such that there exists a random variable $\xi$ on $\left(\Omega, \Sigma, \frac{1}{\mu(\Omega)}\,\mu\right)$, 
which is independent of  $\mathcal{G}$ and has a nontrivial Gaussian distribution. Then 
the corresponding conditional expectation operator $\mathbf{E}^\mathcal{G} = \mathbf{ E}(\cdot | \mathcal{G}) \in \mathcal{B}(L^p(\mu))$ is narrow
(see \cite[Corollary 4.25]{PR13}), but not compact if $\mathcal{G}$ has infinitely many pairwise disjoint elements of positive measure.

Let
\begin{equation}\label{Cp}
C_p := \max_{0 \le \alpha \le1}  
 \left(\alpha^{p - 1} + (1 - \alpha)^{p - 1}\right)^{\frac1p} \left(\alpha^{\frac1{p - 1}} + (1 - \alpha)^{\frac1{p - 1}}\right)^{1 - \frac1p} 
\end{equation}
for $ 1 < p < \infty$, and $C_1 := 2$.

In the following theorems, $(\Omega, \Sigma, \mu)= ([0, 1], \mathcal{L}, \lambda)$, where $\lambda$ is the standard Lebesgue measure on $[0, 1]$
and $\mathcal{L}$ is the $\sigma$-algebra of Lebesgue measurable subsets of $[0, 1]$. 

Our starting point is a result due to C. Franchetti. 
\begin{theorem}[\cite{F90}, \cite{F92}]\label{Fran}
Let $P \in \mathcal{B}(L^p([0, 1]))\setminus\{0\}$ be a narrow projection operator, $1 \le p < \infty$. Then
\begin{equation}\label{P}
\|I - P\|_{L^p \to L^p} \ge \|I - \mathbf{ E}\|_{L^p \to L^p} = C_p .
\end{equation}
\end{theorem}
The following theorem was proved in \cite{SS}, where it was used to show that $C_p$ is the optimal constant in 
the bounded compact approximation property of $L^p([0, 1])$.  It implies the inequality in \eqref{P} in the case when $P \not= 0$ is a finite-rank projection. 
\begin{theorem}[\cite{SS}]\label{compT}
Let $1 \le p < \infty$, $\gamma \in \mathbb{C}$, and let $T \in \mathcal{K}(L^p([0, 1]))$. Then 
\begin{equation}\label{gammag}
\|I - T\|_{L^p \to L^p} + \inf_{\|u\|_{L^p} = 1} \|(\gamma I - T)u\|_{L^p} \ge \|I -  \gamma\mathbf{ E}\|_{L^p \to L^p} .
\end{equation}
In particular,
\begin{equation}\label{gamma1}
\|I - T\|_{L^p \to L^p} + \inf_{\|u\|_{L^p} = 1} \|(I - T)u\|_{L^p} \ge  C_p .
\end{equation}
\end{theorem}

The following theorem is the main result of the paper. It generalises both Theorems \ref{Fran} and \ref{compT}.
\begin{theorem}\label{gammaT}
Estimates \eqref{gammag} and \eqref{gamma1} hold for all narrow operators $T \in \mathcal{B}(L^p([0, 1]))$, $1 \le p < \infty$.
\end{theorem}
In the case $p = 1$, inequality \eqref{gamma1} (for all narrow operators) follows from a Daugavet-type result due to A.M. Plichko and M.M. Popov 
(\cite[\S9, Theorem 8]{PP90}, see also \cite[Corollary 6.4]{PR13}):
$$
\|I - T\| = 1 + \|T\| \quad\mbox{for every narrow operator}\quad T \in \mathcal{B}(L^1([0, 1])) .
$$
Indeed,
\begin{align*}
\|I - T\|_{L^1 \to L^1} + \inf_{\|u\|_{L^1} = 1} \|(I - T)u\|_{L^1} \ge \|I - T\|_{L^1 \to L^1} + 1 - \sup_{\|u\|_{L^1} = 1} \|Tu\|_{L^1}  \\
= 1 + \|T\| _{L^1 \to L^1} + 1 - \|T\|_{L^1 \to L^1}  = 2 = C_1 .
\end{align*}

\section{Proof of Theorem \ref{gammaT}}

It follows from the definition of a narrow operator that if $T \in \mathcal{B}(L^p(\mu))$ is narrow and $S \in \mathcal{B}(L^p(\mu))$, then
$ST \in \mathcal{B}(L^p(\mu))$ is narrow (see \cite[Proposition 1.8]{PR13}). On the other hand, there are $S, T \in \mathcal{B}(L^p(\mu))$ 
such that $T$ is narrow but $TS$ is not (see \cite[Proposition 5.1]{PR13}). The following lemma shows that the latter cannot happen 
if $S$ is a multiplication operator. 
\begin{lemma}\label{sign}
Let $X = L^p(\mu)$,\, $g \in L^\infty(\mu)$,\, and $T \in \mathcal{B}(X, Y)$ be a narrow operator. Then the operator $TgI \in \mathcal{B}(X, Y)$ is also narrow.
\end{lemma}
\begin{proof}
There is nothing to prove if $T = 0$. Suppose $T \not= 0$.
Take any $A \in \Sigma^+$ and any $\varepsilon > 0$.
There exists a simple function
$g_0 =\sum_{k = 1}^M a_k \mathbb{I}_{B_k} \not\equiv 0$ such that
$$
\|g - g_0\|_{L^\infty} < \frac{\varepsilon}{2 \|T\| (\mu(\Omega))^{1/p}}\, .
$$
Here  $M \in \mathbb{N}$,\, $B_k \in \Sigma$, $k = 1, \dots, M$
are pairwise disjoint, $\mu(B_k) > 0$, $\cup_{k = 1}^M B_k = \Omega$,\, $a_k \in \mathbb{C}$.

Let $A_k := A\cap B_k$. If $\mu(A_k) > 0$, let $h_k$ be a mean zero sign on $A_k$ such that
$$
\|Th_k\| < \frac{\varepsilon}{2\sum_{k = 1}^M |a_k|}\, . 
$$ 
Let 
$$
h := \sum_{\{k : \ \mu(A_k) > 0\}} h_k .
$$
It is clear that $h$ is a mean zero sign on $A$ and
\begin{align*}
\|TgI h\| & \le \|Tg_0 h\| +  \|T(g - g_0) h\| \\
& \le  \left\|T\sum_{\{k : \ \mu(A_k) > 0\}} a_k h_k\right\|+ \|T\| \|g - g_0\|_{L^\infty} \|h\|_{L^p} \\
& < \sum_{\{k : \ \mu(A_k) > 0\}} |a_k| \|Th_k\| + \|T\|\, \frac{\varepsilon}{2 \|T\| (\mu(\Omega))^{1/p}}\, (\mu(\Omega))^{1/p} \\
&\le \frac{\varepsilon}{2\sum_{k = 1}^M |a_k|} \sum_{\{k : \ \mu(A_k) > 0\}} |a_k| + \frac{\varepsilon}{2} \le 
\frac{\varepsilon}{2}  + \frac{\varepsilon}{2}  = \varepsilon .
\end{align*}
\end{proof}
The above lemma and its proof remain valid if $X$ is a K\"othe F-space on $(\Omega, \Sigma, \mu)$ (see \cite[Section 1.3]{PR13}).
Similarly, the following lemma and its proof remain valid if $X$ is a rearrangement-invariant Banach space on $([0, 1], \mathcal{L}, \lambda)$
with absolutely continuous norm.
This lemma is a minor modification of \cite[Theorem 2.21]{PR13} and \cite[\S 8, Proposition 5]{PP90}.
\begin{lemma}\label{one}
Let $X = L^p([0, 1])$  and $T \in \mathcal{B}(X, Y)$ be a narrow operator.
Then there exists a 1-complemented subspace $X_0$ of $X$ isometrically isomorphic to $X$ such that $\mathbf{1}  \in X_0$ and the restriction $T|_{X_0}$ of $T$
to $X_0$ is a compact operator. 
\end{lemma}
\begin{proof}
Take a mean zero sign $\overline{g}_1$ on $[0, 1]$ and set $T_1 := T\overline{g}_1 I$. The operator $T_1$ is narrow according to Lemma \ref{sign}.
The proof of \cite[Theorem 2.21]{PR13} (with $T_1$ in place of $T$ and with $\varepsilon \ge 2\|T_1 \overline{g}_1\|,\, \varepsilon_1 > \|T_1 \overline{g}_1\|$)
shows that there exists a 1-complemented (see the proof of \cite[\S 8, Proposition 5]{PP90}) 
subspace $X_1$ of $X$ isometrically isomorphic to $X$ such that $\overline{g}_1  \in X_1$ and the restriction $T_1|_{X_1}$ of $T_1$
to $X_1$ is a compact operator. Let $X_0 := \overline{g}_1 X_1$. Since $\overline{g}_1^2 = \mathbf{1}$, the operator of multiplication $\overline{g}_1 I$ is an 
isometric isomorphism of $X_0$ onto $X_1$ and of $X$ onto itself. Let $P_1 \in \mathcal{B}(X)$ be a projection onto $X_1$ such that $\|P_1\| = 1$.
Then $P_0 := \overline{g}_1 P_1 \overline{g}_1 I \in \mathcal{B}(X)$ is a projection onto $X_0$ such that $\|P_0\| = 1$. Hence $X_0$ is 1-complemented
(this follows also from  \cite[Theorem 4]{A66}, since $X_0$ is isometrically isomorphic to $X = L^p([0, 1])$).
Further, $\mathbf{1} = \overline{g}_1^2 \in \overline{g}_1 X_1 = X_0$ and
$T|_{X_0} = T_1|_{X_1}\, \overline{g}_1 I|_{X_0}$ is compact.
\end{proof}

\begin{proof}[Proof of Theorem \ref{gammaT}]
Take an arbitrary $\varepsilon > 0$. Let
\begin{equation}\label{delta}
\delta := \inf_{\|u\|_{L^p} = 1} \|(\gamma I - T)u\|_{L^p} .
\end{equation}
There exists $u_0  \in L^p([0, 1])$ such that $\|u_0\|_{L^p} = 1$ and $\|(\gamma I - T)u_0\|_{L^p} < \delta + \epsilon$. Then
there exists an approximation $h :=\sum_{k = 1}^M a_k \mathbb{I}_{A_k}$ of $u_0$ such that $A_k$, $k = 1, \dots, M$, $M \in \mathbb{N}$
are pairwise disjoint Borel sets of positive measure, $\cup_{k = 1}^M A_k = [0, 1]$, $a_k \in \mathbb{C}$, and
\begin{equation}\label{deltaeps}
\left\|\gamma h - Th\right\|_{L^p([0, 1])} \le \delta + 2\varepsilon , \quad \|h\|_{L^p([0, 1])} = 1 .
\end{equation}
Partition $[0, 1]$ into subintervals $I_k$ of length $\lambda(A_k)$, $k = 1, \dots, M$.
Since $(A_k, \mathcal{L}, \lambda)$
is isomorphic (modulo sets of measure $0$) to $(I_k, \mathcal{L}, \lambda)$ (see, e.g., \cite[Theorem 9.2.2 and Corollary 6.6.7]{B07}),
one can easily derive from Lemma \ref{one} the existence, for each $k$, of a 1-complemented subspace $X_k$ of $L^p(A_k)$ isometrically isomorphic to $L^p(I_k)$ 
such that $\mathbb{I}_{A_k}  \in X_k$ and $T|_{X_k}$ is a compact operator. 
Let
$$
X_0 := \left\{f \in L^p([0, 1]) : \ f|_{A_k} \in X_k , \, k = 1, \dots, M \right\}.
$$
It is easy to see that $X_0$ is 1-complemented and isometrically isomorphic to $L^p([0, 1])$, and that $T_0 := T|_{X_0}$ is a compact operator. 
Let $J : L^p([0, 1]) \to X_0$ be an isometric isomorphism and $P_0 \in \mathcal{B}(L^p([0, 1]))$ be a projection onto $X_0$ such that $\|P_0\| = 1$.
Then $T_1 := J^{-1}P_0 T_0 J \in \mathcal{K}(L^p([0, 1]))$ and it follows from Theorem \ref{compT} that
\begin{align*}
& \|I - T_0\|_{X_0 \to L^p} + \inf_{f \in X_0, \, \|f\|_{L^p} = 1} \|(\gamma I - T_0)f\|_{L^p} \\
& \ge \|P_0(I - T_0)\|_{X_0 \to X_0} + \inf_{f \in X_0, \, \|f\|_{L^p} = 1} \|P_0(\gamma I - T_0)f\|_{L^p} \\
& = \|J^{-1}P_0(I - T_0)J\|_{L^p \to L^p} + \inf_{\varphi \in L^p, \, \|\varphi\|_{L^p} = 1} \|J^{-1}P_0(\gamma I - T_0)J\varphi\|_{L^p} \\
& = \|I - T_1\|_{L^p \to L^p} + \inf_{\varphi \in L^p, \, \|\varphi\|_{L^p} = 1} \|(\gamma I - T_1)\varphi\|_{L^p} 
\ge \|I -  \gamma\mathbf{ E}\|_{L^p \to L^p} .
\end{align*}
Since $h \in X_0$, it follows from \eqref{deltaeps} that
$$
\delta + 2\varepsilon \ge \inf_{f \in X_0, \, \|f\|_{L^p} = 1} \|(\gamma I - T_0)f\|_{L^p} .
$$
Hence
\begin{align*}
\label{}
\|I - T\|_{L^p \to L^p} + \delta + 2\varepsilon & \ge \|I - T_0\|_{X_0 \to L^p} + \inf_{f \in X_0, \, \|f\|_{L^p} = 1} \|(\gamma I - T_0)f\|_{L^p} \\
&\ge \|I -  \gamma\mathbf{ E}\|_{L^p \to L^p} 
\end{align*}
and
$$
\|I - T\|_{L^p \to L^p} + \inf_{\|u\|_{L^p} = 1} \|(\gamma I - T)u\|_{L^p} + 2\varepsilon \ge \|I -  \gamma\mathbf{ E}\|_{L^p \to L^p} 
$$
for all $\varepsilon > 0$ (see \eqref{delta}).
\end{proof}


\begin{thebibliography}{1}

\bibitem{A66} 
T. Ando, 
Contractive projections in $L_p$-spaces,
{\em Pacific J. Math.} \textbf{17}, 391--405, 1966.

\bibitem{B07}
V.I. Bogachev, 
\textit{Measure theory. Vol. I and II.} 
Springer, Berlin, 2007.

\bibitem{F90}
C. Franchetti, 
The norm of the minimal projection onto hyperplanes in $L^p[0,1]$
 and the radial constant,
\textit{Boll. Unione Mat. Ital.}, \textbf{VII}, Ser., B 4, 4, 803--821, 1990.

\bibitem{F92}
C. Franchetti, 
Lower bounds for the norms of projections with small kernels, 
\textit{Bull. Aust. Math. Soc.} \textbf{45}, 3, 507--511, 1992.


\bibitem{PP90}
A.M. Plichko and M.M. Popov, 
Symmetric function spaces on atomless probability spaces.
\textit{Diss. Math.} \textbf{306}, 85 p., 1990.

\bibitem{PR13}
M. Popov and B. Randrianantoanina, 
\textit{Narrow operators on function spaces and vector lattices.} 
de Gruyter Studies in Mathematics \textbf{45}, de Gruyter, Berlin, 2013.

\bibitem{SS}
E. Shargorodsky and T. Sharia,
Sharp estimates for conditionally  centred moments 
and for compact operators on  $L^p$ spaces,
(to appear).



\end{thebibliography}
\end{document}